\documentclass{elsarticle}

\usepackage{amscd}
\usepackage{amsmath}
\usepackage{amssymb}
\usepackage{amsthm}
\usepackage{graphicx}
\usepackage[all]{xy}
\usepackage{kotex}
\usepackage{color}
\usepackage{enumerate}
\theoremstyle{plain}

\usepackage{nameref}

\newtheorem{Thm}{Theorem}[section]
\newtheorem{Lemma}[Thm]{Lemma}

\newtheorem{Cor}[Thm]{Corollary}
\newtheorem{definition}[Thm]{Definition}
\newtheorem*{Thm*}{Theorem}
\theoremstyle{remark}
\newtheorem{Rmk}[Thm]{\bf{Remark}}

\usepackage{color}							

\journal{Computers \& Mathematics with Applications}
\begin{document}
\begin{frontmatter}
\title{
A New Condition for Blow-up Solutions to Discrete Semilinear Heat Equations on Networks\\
}

\author[syc]{Soon-Yeong Chung}
\address[syc]{Department of Mathematics and Program of Integrated Biotechnology, Sogang University, Seoul 04107, Korea}
\ead{sychung@sogang.ac.kr}

\author[mjc]{Min-Jun Choi\corref{cjhp}}
\address[mjc]{Department of Mathematics, Sogang University, Seoul 04107, Korea}
\cortext[cjhp]{Corresponding author}
\ead{dudrka2000@sogang.ac.kr}

\begin{abstract}

The purpose of this paper is to introduce a new condition
\begin{center}
(C)$\hspace{1cm} \alpha \int_{0}^{u}f(s)ds \leq uf(u)+\beta u^{2}+\gamma,\,\,u>0$
\end{center}
for some $\alpha, \beta, \gamma>0$ with $0<\beta\leq\frac{\left(\alpha-2\right)\lambda_{0}}{2}$, where $\lambda_{0}$ is the first eigenvalue of discrete Laplacian $\Delta_{\omega}$, with which we obtain blow-up solutions to discrete semilinear heat equations
\begin{equation*}
\begin{cases}
u_{t}\left(x,t\right)=\Delta_{\omega}u\left(x,t\right)+f(u(x,t)), & \left(x,t\right)\in S\times\left(0,+\infty\right),\\
u\left(x,t\right)=0, & \left(x,t\right)\in\partial S\times\left[0,+\infty\right),\\
u\left(x,0\right)=u_{0}\geq0(nontrivial), & x\in\overline{S}
\end{cases}
\end{equation*}
on a discrete network $S$. In fact, it will be seen that the condition (C) improves the conditions known so far.
\end{abstract}

\end{frontmatter}

\section{Introduction}

These days, the reaction-diffusion systems have found many applications ranging from chemical and biological phenomena to medicine, genetics, and so on. A typical example of the reaction-diffusion system is an auto-catalytic chemical reaction between several chemicals in which the concentration of each chemical grows (or decays) due to diffusion and difference of concentration (according to Fick$'$s law, for example) and whose phenomena is modeled by the reaction-diffusion system
\begin{equation}\label{RDS}
u_{t}\left(x,t\right)={\displaystyle \sum_{x\in\overline{S}}\left[u\left(y,t\right)-u\left(x,t\right)\right]\omega\left(x,y\right)+u^{q}\left(x,t\right),\:} (x,t)\in S\times(0,\infty)
\end{equation}
with some boundary and initial conditions where $S$ is the set of chemicals.

From a similar point of view, we discuss,in this paper, the blow-up property of solutions to the following discrete semilinear heat equations
\begin{equation}\label{inteq}
\left\{
  \begin{array}{ll}
    u_{t}(x,t) = \Delta_{\omega}u(x,t) + f(u(x,t)), & \hbox{$(x,t)\in S\times (0, +\infty)$,} \\
    u(x,t)=0, & \hbox{$(x,t)\in \partial S \times [0, +\infty)$,} \\
    u(x, 0)=u_{0}(x)\geq 0, & \hbox{$x\in \overline{S},$}
  \end{array}
\right.
\end{equation}
which generalizes the equation \eqref{RDS} and where $\Delta_{\omega}$  denotes the discrete Laplacian operator (which will be introduced in Section \ref{preli}).

The continuous case of this equation has been studied by many authors. For example, in $1973$, Levine \cite{L} considered the formally parabolic equations of the form
\[
\begin{cases}
P\frac{du}{dt}=-A(t)u + f(u(t)), & t\in [0, +\infty),\\
u(0)=u_{0},
\end{cases}
\]
where $P$ and $A(t)$ are positive linear operators defined on a dense subdomain $D$ of a real or complex Hilbert space $H$. Here, he first introduced the concavity method and proved that there exists a time $T$ such that
\[
\lim_{t\rightarrow T^{-}}\int_{0}^{t}\int_{\Omega}u(x,s)P(u(x,s))dxds=+\infty,
\]
under the condition
\[
\begin{aligned}
\hbox{(A) \hspace{1cm} $(2+\epsilon)F(u)\leq uf(u),\,\,u>0$},
\end{aligned}
\]
for some $\epsilon>0$ and the initial data $u_{0}$ satisfying
\[
\frac{1}{2}\int_{\Omega}u_{0}(x)\cdot A(0)[u_{0}(x)]dx + \int_{\Omega}F(u_{0}(x))dx>0,
\]
where $F(u)=\int_{0}^{u}f(s)ds$.

After this, Philippin and Proytcheva \cite{PP} have applied the above method to the equations
\begin{equation}\label{IBBequatio}
\begin{cases}
u_{t}=\Delta u + f(u), & \hbox{in}\,\,\Omega\times(0,+\infty),\\
u(x,t)=0, & \hbox{on}\,\, \partial \Omega\times(0,+\infty),\\
u(x,0)=u_{0}(x)\geq0,
\end{cases}
\end{equation}
and obtained a blow-up solution, under the condition (A) and the initial data $u_{0}$ satisfying
\[
-\frac{1}{2}\int_{\Omega}|\nabla u_{0}(x)|^{2}dx + \int_{\Omega}F(u_{0}(x))dx>0.
\]
Recently, Ding and Hu \cite{DH} adopted the condition (A) to get blow-up solutions to the equation
\[
(g(u))_{t}=\nabla\cdot(\rho(|\nabla u|^{2})\nabla u) + k(t)f(u)
\]
with the nonnegative initial value and the null Drichlet boundary condition.

Besides, in \cite{PS, PPP} Payne et al. obtained the blow-up solutions to the equations
\begin{equation}\label{IBBequation}
\begin{cases}
u_{t}=\Delta u - g(u), & \hbox{in}\,\,\Omega\times(0,+\infty),\\
\frac{\partial u}{\partial n}=f(u), & \hbox{on}\,\, \partial \Omega\times(0,+\infty),\\
u(x,0)=u_{0}(x)\geq0,
\end{cases}
\end{equation}
when the Neumann boundary data $f$ satisfies the condition (A).

On the other hands, the condition (A) was relaxed by Bandle and Brunner \cite{BB} and has been applied to the equations
\begin{equation}\label{IBBequation}
\begin{cases}
u_{t}=\Delta u + f(x,t,u,\nabla u), & \hbox{in}\,\,\Omega\times(0,+\infty),\\
u(x,t)=0, & \hbox{on}\,\, \partial \Omega\times(0,+\infty),\\
u(x,0)=u_{0}(x)\geq0.
\end{cases}
\end{equation}
In fact, they introduced a condition
\[
\begin{aligned}
\hbox{(B) \hspace{1cm} $(2+\epsilon)F(u)\leq uf(u) + \gamma ,\,\,u>0$}
\end{aligned}
\]
and derived the blow-up solutions to the equation \eqref{IBBequation}, under the condition (B) and the initial data $u_{0}$ satisfying
\[
-\frac{1}{2}\int_{\Omega}|\nabla u_{0}(x)|^{2}dx+\int_{\Omega}[F(x,u_{0})-\gamma]dx>0,
\]
for some $\epsilon>0$. \\

Looking into the concavity method more closely, we can see that the proof consists of a series of inequalities with reasoning and the Poincare inequality including the eigenvalue. But the conditions (A) and (B) above are independent of the eigenvalue which depends on the domain. From this observation, we can expect to develop an improved condition which refines (A) or (B), depending on the domain. Being motivated by this point of view, we develop a new condition as follows: for some $\alpha, \beta, \gamma>0$,
\[
\hbox{(C)\hspace{1cm} $\alpha F\left(u\right)\leq uf(u) + \beta u^{2} + \gamma,\,\,u>0$},
\]
where $0<\beta\leq\frac{\left(\alpha-2\right)\lambda_{0}}{2}$ and $\lambda_{0}$ is the first eigenvalue of the discrete Laplacian $\Delta_{\omega}$. Here, we note that the term $\beta u^{2}$ is depending on the domain graph.

In fact, it is expected that, with the condition (C), more interesting results should be obtained even in the continuous case, which will be our forth-coming work.

The blow-up solutions or global existence to the discrete equation \eqref{inteq} with the case $f(u)=u^{q}$, was already studied in \cite{CL} and \cite{XXM}.
Moreover, in \cite{CC} and \cite{Ch}, the authors gave a complete solutions under general case of Laplacian ($p$-Laplacian) with $f(u)=u^{q}$. Besides, the long time behavior (extinction and positivity) of solutions to discrete evolution Laplace equation with absorption on networks was studied in \cite{CLC} and \cite{LC}.\\

The main theorem of this paper is as follows:
\begin{Thm*}[Concavity Method]
For the function $f$ with the hypothesis (C), if the initial data $u_{0}$ satisfies
\begin{equation}
-\frac{1}{4}\sum_{x,y\in\overline{S}}\left[u_{0}\left(x\right)-u_{0}\left(y\right)\right]^{2}\omega\left(x,y\right)+\sum_{x\in\overline{S}}\left[F(u_{0}\left(x\right))-\gamma\right]>0,
\end{equation}
then the solutions $u$ to the equation \eqref{inteq} blow up at finite time $T^{*}$ in the sense of
\[
\lim_{t\rightarrow T^{*}}\int_{0}^{t}\sum_{x\in\overline{S}}u^{2}\left(x,s\right)ds=+\infty,
\]
where $\gamma$ is the constant in the condition (C).
\end{Thm*}

We organize this paper as follows: in Section \ref{preli}, we introduce briefly the preliminary concepts on networks and comparison principles. Section \ref{BCM} is the main section, which is devoted to blow-up solutions using the concavity method with the condition (C). Finally in Section \ref{(C)}, we discuss the condition (C), comparing with the conditions (A) and (B), together with the condition   $J(0)>0$ for the initial data.


\section{Preliminaries and Discrete Comparison Principles}\label{preli}
In this section, we start with the theoretic graph notions frequently used throughout this paper. For more detailed information on notations, notions, and conventions, we refer the reader to \cite{CB}.

\begin{definition}
\begin{enumerate}[(i)]
\item  A graph $G=G\left(V,E\right)$ is a finite set $V$ of $vertices$ with a set $E$ of $edges$ (two-element subsets of $V$). Conventionally used, we denote by $x\in V$ or $x\in G$ the fact that $x$ is a vertex in $G$.
\item A graph $G$ is called $simple$ if it has neither multiple edges nor loops
\item $G$ is called $connected$ if for every pair of vertices $x$ and $y$, there exists a sequence(called a $path$) of vertices $x=x_{0},x_{1},\cdots,x_{n-1},x_{n}=y$ such that $x_{j-1}$ and $x_{j}$ are connected by an edge(called $adjacent$) for $j=1,\cdots,n$.
\item A graph $G'=G'\left(V',E'\right)$ is called a $subgraph$ of $G\left(V,E\right)$ if $V'\subset V$ and $E'\subset E$. In this case, $G$ is a host graph of $G'$. If $E'$ consists of all the edges from $E$ which connect the vertices of $V'$ in its host graph $G$, then $G'$ is called an induced subgraph.
\end{enumerate}
\end{definition}
 We note that an induced subgraph of a connected host graph may not be connected.

Throughout this paper, all the subgraphs are assumed to be induced, simple and connected.

\begin{definition}
A $weight$ on a graph $G$ is a symmetric function $\omega\,:\,V\times V\rightarrow\left[0,+\infty\right)$ satisfying the following:
\begin{enumerate}[(i)]
\item $\omega\left(x,x\right)=0$, \,\,$x\in V$,
\item $\omega\left(x,y\right)=\omega\left(y,x\right)$ if $x\sim y$,
\item $\omega\left(x,y\right)>0$ if and only if $\{x,y\}\in E$,
\end{enumerate}
and a graph $G$ with a weight $\omega$ is called a $network$.
\end{definition}

\begin{definition}
For an induced subgraph $S$ of a $G=G\left(V,E\right)$, the (vertex) $boundary$ $\partial S$ of $S$ is defined by
\[
\partial S:=\{z\in V\setminus S\,|\, \hbox{$z\sim y$ for some $y\in S$}\}.
\]
\end{definition}

Also, we denote by $\overline{S}$ a graph whose vertices and edges are in $S\cup\partial S$. We note that by definition the set $\overline{S}$ is an induced subgraph of $G$.

\begin{definition}
The \emph{degree} $d_{\omega}x$ of a vertex $x$ in a network $S$ (with boundary $\partial S$) is defined by
\[
d_{\omega}x:=\sum_{y\in\overline{S}}\omega\left(x,y\right).
\]
\end{definition}

\begin{definition}
For a function $u\,:\,\overline{S}\rightarrow\mathbb{R}$, the \emph{discrete Laplacian} $\Delta_{\omega}$ on $S$ is defined by
\[
\Delta_{\omega}u\left(x\right):=\sum_{y\in\overline{S}}\left[u\left(y\right)-u\left(x\right)\right]\omega\left(x,y\right)
\]
for $x\in S$.
\end{definition}

The following two lemmas are used throughout this paper.
\begin{Lemma}[See \cite{PC}, \cite{PaC}]\label{Lemma 2.1.}
For functions $f,~g\,:\,\overline{S}\to\mathbb{R}$,
the discrete Laplacian $\Delta_\omega$ satisfies that
$$
2\sum_{x\in\overline{S}}g\left(x\right)\left[-\Delta_{\omega}f\left(x\right)\right]
=\sum_{x,y\in\overline{S}}\left[f\left(y\right)-f\left(x\right)\right]\cdot\left[g\left(y\right)-g\left(x\right)\right]\omega\left(x,y\right).
$$
In particular, in the case $g=f$, we have
$$
2\sum_{x\in\overline{S}}f\left(x\right)\left[-\Delta_{\omega}f\left(x\right)\right]=\sum_{x,y\in\overline{S}}\left[f\left(x\right)-f\left(y\right)\right]^{2}\omega\left(x,y\right).
$$
\end{Lemma}
\begin{Lemma}[See \cite{PC}, \cite{PaC}]\label{Lemma 2.2.}

There exist $\lambda_{0}>0$ and $\phi_{0}\left(x\right)>0$, $x\in S$
such that
\[
\begin{cases}
-\Delta_{\omega}\phi_{0}\left(x\right)=\lambda_{0}\phi_{0}\left(x\right), & x\in S,\\
\phi_{0}\left(x\right)=0, & x\in\partial S,\\
{\displaystyle \sum_{x\in S}}\phi_{0}^{2}\left(x\right)=1.
\end{cases}
\]
Moreover, $\lambda_{0}$ is given by
\[
\begin{aligned}
\lambda_{0} =& \min_{u\in\mathcal{A},u\not\equiv0}\frac{\frac{1}{2}{\displaystyle\sum_{x,y\in\overline{S}}}\left[u\left(x\right)-u\left(y\right)\right]^{2}\omega\left(x,y\right)}{{\displaystyle\sum_{x\in\overline{S}}}\left[u\left(x\right)\right]^{2}}\\
\leq & \,\,d_{\omega}x, x\in S
\end{aligned}
\]
where $\mathcal{A}:=\left\{ u\,:\,\overline{S}\rightarrow\mathbb{R}\,|\, u=0\mbox{ on }\partial S\right\}$.
\end{Lemma}
In the above, the number $\lambda_{0}$ is called the first eigenvalue of $\Delta_{\omega}$ on a network $\overline{S}$ with corresponding eigenfunction $\phi_{0}$ (see \cite{C} and \cite{CDS} for the spectral theory of the Laplacian operators).


We now briefly discuss the local existence and uniqueness of a solution for the equation
\begin{equation}\label{equation}
\begin{cases}
u_{t}\left(x,t\right)=\Delta_{\omega}u\left(x,t\right)+f\left(u\left(x,t\right)\right), & \left(x,t\right)\in S\times\left(0,+\infty\right),\\
u\left(x,t\right)=0, & \left(x,t\right)\in\partial S\times\left[0,+\infty\right),\\
u\left(x,0\right)=u_{0}\left(x\right)\geq0, & x\in\overline{S},
\end{cases}
\end{equation}
where $f$ be locally Lipschitz continuous on $\mathbb{R}$.

Let $t_{0}>0$ be fixed and consider a Banach space
\[
X_{t_{0}}=\{u\,:\,\overline{S}\times\left[0,t_{0}\right]\rightarrow\mathbb{R}\,|\,u\left(x,\cdot\right)\in C\left[0,t_{0}\right]\hbox{ for each }x\in\overline{S}\}
\]
with the norm $\|u\|_{X_{t_{0}}}:=\max_{x\in\overline{S}}\max_{0\leq t\leq t_{0}}\left|u\left(x,t\right)\right|$.

Then it is easy to see that the operator $D\,:\,X_{t_{0}}\rightarrow X_{t_{0}}$ given by
\[
D\left[u\right]\left(x,t\right):=
\begin{cases}
u_{0}\left(x\right)+\int_{0}^{t}\Delta_{\omega}u\left(x,s\right)ds+\int_{0}^{t}f\left(u\left(x,s\right)\right)ds, & \left(x,t\right)\in S\times\left[0,t_{0}\right],\\
0, & \left(x,t\right)\in\partial S\times\left[0,t_{0}\right],
\end{cases}
\]
is well-defined. In Lemma \ref{contraction}, we show that this operator $D$ is contractive on a small closed ball. Hence, we obtain the existence and uniqueness of a solution to the equation \eqref{equation} in a small time interval $\left[0,t_{0}\right]$, as a consequence of Banach$'$s fixed point theorem.
\begin{Lemma}\label{contraction}
Let $f$ be locally Lipschitz continuous on $\mathbb{R}$. Then the operator $D$ is a contraction on the closed ball
\[
B(u_{0}, 2\|u_{0}\|_{X_{t_{0}}}) := \{ u \in X_{t_{0}} | \parallel u - u_{0} \parallel_{X_{t_{0}}} \leq 2\|u_{0}\|_{{X_{t_{0}}}} \}
\]
if $t_{0}$ is small enough.
\begin{proof}
Consider $u$ and $v \in B(u_{0}, 2\|u_{0}\|_{{X_{t_{0}}}})$. Since $f$ is locally Lipschitz continuous on $\mathbb{R}$, there exists $L>0$ such that
\[
|f\left(a\right)-f\left(b\right)|\leq L|a-b|,\,\,a,b\in[-m,m]
\]
where $m=3\|u_{0}\|_{{X_{t_{0}}}}.$ Then for any $(x,t) \in S \times [0,t_{0}]$,
\[
\begin{aligned}
&|D[u](x,t) - D[v](x,t)|\\
&=\left|\int_{0}^{t} \Delta_{\omega}(u-v)(x,s)ds + \int_{0}^{t}\left[f(u(x,s)) - f(v(x,s))\right]ds\right|\\
&\leq 2t|\overline{S}|\max_{(x,y)\in E} \omega(x,y)\|u-v\|_{X_{t_{0}}} + Lt\|u-v\|_{X_{t_{0}}}\\
&\leq C_{1}t_{0}\|u-v\|_{X_{t_{0}}},
\end{aligned}
\]
where $C_{1} = 2|\overline{S}|  \max_{(x,y)\in E}\omega(x,y) + L$ and $|\overline{S}|$ denotes the number of nodes in $\overline{S}$. Moreover, it is easy to see that the above inequality still holds for $(x,t) \in \partial S \times [0,t_{0}]$. Hence choosing $t_{0}$ sufficiently small, we obtain a contraction on the closed ball $B(u_{0}, 2\|u_{0}\|_{{X_{t_{0}}}})$ into itself.
\end{proof}
\end{Lemma}

Now, we state two types of comparison principles.
\begin{Thm}[Comparison Principle]\label{comparison principle}
Let $T>0$ ($T$ may be $+\infty$) and $f$ be locally Lipschitz continuous on $\mathbb{R}$. Suppose that real-valued functions $u(x,\cdot)$, $v(x,\cdot) \in C[0, T]$ are differentiable in $(0, T)$ for each $x \in \overline{S}$ and satisfy
\begin{equation}\label{eq1-2}
\begin{cases}
 u_{t}(x,t)-\Delta_{\omega}u(x,t)-f(u(x,t)) & \\
\geq v_{t}(x,t)-\Delta_{\omega}v(x,t)-f(v(x,t)), & (x,t)\in S\times\left(0,T\right),\\
 u(x,t)\geq v(x,t), & (x,t)\in\partial S\times[0,T),\\
 u\left(x,0\right)\geq v\left(x,0\right), & x\in\overline{S}.
\end{cases}
\end{equation}
Then $u\left(x,t\right)\geq v\left(x,t\right)$ for all $(x,t)\in\overline{S}\times[0,T).$
\end{Thm}
\begin{proof}
Let $T'>0$ be arbitrarily given with $T'<T$. Since $f$ is locally Lipschitz continuous on $\mathbb{R}$, there exists $L>0$ such that
\begin{equation}\label{eq1-3}
\left|f\left(a\right)-f\left(b\right)\right|\leq L\left|a-b\right|,\,\,a,b\in[-m,m]
\end{equation}
where $m={\displaystyle \max_{x\in\overline{S}}\max_{0\leq t\leq T'}}\left\{\left|u\left(x,t\right)\right|,\left|v\left(x,t\right)\right|\right\}.$
Let $\tilde{u}, \tilde{v}\,:\,\overline{S}\times\left[0,T'\right]\rightarrow\mathbb{R}$ be the functions defined by
\[
\tilde{u}(x,t):=e^{-2Lt}u\left(x,t\right),\,\,
(x,t)\in\overline{S}\times[0,T'].
\]
\[
\tilde{v}(x,t):=e^{-2Lt}v\left(x,t\right),\,\,
(x,t)\in\overline{S}\times[0,T'].
\]

Then from \eqref{eq1-2}, we have
\begin{equation}\label{eq1-4}
\begin{aligned}
&\left[\tilde{u}_{t}\left(x,t\right)-\tilde{v}_{t}\left(x,t\right)\right]-\left[\Delta_{\omega}\tilde{u}\left(x,t\right)-\Delta_{\omega}\tilde{v}\left(x,t\right)\right]\\ &+2L\left[\tilde{u}\left(x,t\right)-\tilde{v}\left(x,t\right)\right]-e^{-2Lt}\left[f\left(u\left(x,t\right)\right)-f\left(v\left(x,t\right)\right)\right]\geq0
\end{aligned}
\end{equation}
for all $\left(x,t\right)\in S\times(0,T']$.

We recall that $\tilde{u}(x,\cdot)$ and $\tilde{v}(x,\cdot)$ are continuous on $[0,T']$ for each $x\in \overline{S}$ and $\overline{S}$ is finite. Hence, we can find
$\left(x_{0},t_{0}\right)\in\overline{S}\times\left[0,T'\right]$ such that
\[
\left(\tilde{u}-\tilde{v}\right)\left(x_{0},t_{0}\right)={\displaystyle \min_{x\in\overline{S}}\min_{0\leq t\leq T'}\left(\tilde{u}-\tilde{v}\right)\left(x,t\right)},
\]
which implies that
\begin{equation}\label{123}
\tilde{v}\left(y,t_{0}\right)-\tilde{v}\left(x_{0},t_{0}\right)\leq \tilde{u}\left(y,t_{0}\right)-\tilde{u}\left(x_{0},t_{0}\right),\,\,\,y\in\overline{S}.
\end{equation}

Then now we have only to show that $\left(\tilde{u}-\tilde{v}\right)\left(x_{0},t_{0}\right)\geq0$.

Suppose that $\left(\tilde{u}-\tilde{v}\right)\left(x_{0},t_{0}\right)<0$, on the contrary. Since $\left(\tilde{u}-\tilde{v}\right)\geq0$ on both $\partial S\times\left[0,T'\right]$ and $\overline{S}$, we have $\left(x_{0},t_{0}\right)\in S\times(0,T']$. Then we obtain from \eqref{123} that
\begin{equation}\label{eq1-5}
\Delta_{\omega}\left(\tilde{u}-\tilde{v}\right)\left(x_{0},t_{0}\right)\geq 0
\end{equation}
and it follows from the differentiability of $\left(\tilde{u}-\tilde{v}\right)\left(x,t\right)$ in $(0,T']$ for each $x\in\overline{S}$ that
\begin{equation}\label{eq1-6}
\left(\tilde{u}_{t}-\tilde{v}_{t}\right)\left(x_{0},t_{0}\right)\leq 0.
\end{equation}

According to \eqref{eq1-3}, we have
\begin{equation}\label{eq1-7}
\begin{aligned}
&2L\left[\tilde{u}\left(x_{0},t_{0}\right)-\tilde{v}\left(x_{0},t_{0}\right)\right]-e^{-2Lt_{0}}\left[f\left(u\left(x_{0},t_{0}\right)\right)-f\left(v\left(x_{0},t_{0}\right)\right)\right]\\
&\leq2L\left[\tilde{u}\left(x_{0},t_{0}\right)-\tilde{v}\left(x_{0},t_{0}\right)\right]+Le^{-2Lt_{0}}\left|u\left(x_{0},t_{0}\right)-v\left(x_{0},t_{0}\right)\right|\\
&=2L\left[\tilde{u}\left(x_{0},t_{0}\right)-\tilde{v}\left(x_{0},t_{0}\right)\right]+L\left|\tilde{u}\left(x_{0},t_{0}\right)-\tilde{v}\left(x_{0},t_{0}\right)\right|\\
&=L\left[\tilde{u}\left(x_{0},t_{0}\right)-\tilde{v}\left(x_{0},t_{0}\right)\right]<0,
\end{aligned}
\end{equation}
since $\tilde{u}\left(x_{0},t_{0}\right)<\tilde{v}\left(x_{0},t_{0}\right)$. Combining \eqref{eq1-5}, \eqref{eq1-6}, \eqref{eq1-7}, we obtain the following:
\[
\begin{aligned}
&\tilde{u}\left(x_{0},t_{0}\right)-\tilde{v}\left(x_{0},t_{0}\right)-\left[\Delta_{\omega}\tilde{u}\left(x_{0},t_{0}\right)-\Delta_{\omega}\tilde{v}\left(x_{0},t_{0}\right)\right]\\
&+2L\left[\tilde{u}\left(x_{0},t_{0}\right)-\tilde{v}\left(x_{0},t_{0}\right)\right]-e^{-2Lt_{0}}\left[f\left(u\left(x_{0},t_{0}\right)\right)-f\left(v\left(x_{0},t_{0}\right)\right)\right]<0,
\end{aligned}
\]
which contradicts \eqref{eq1-4}. Therefore $\tilde{u}\left(x,t\right)\geq\tilde{v}\left(x,t\right)$ for all $(x,t)\in S\times(0,T']$ so that we get $u\left(x,t\right)\geq v\left(x,t\right)$ for all $(x,t)\in\overline{S}\times[0,T)$, since $T'<T$ is arbitrarily given.
\end{proof}

\begin{Thm}[Strong Comparison Principle]\label{SCP}
Let $T>0$ $(T\, may\,be +\infty)$ and $f$ be locally Lipschitz continuous on $\mathbb{R}$. Suppose that real-valued functions $u(x,\cdot)$, $v(x,\cdot) \in C[0, T]$ are differentiable in $(0, T)$ for each $x \in \overline{S}$ and satisfy
\begin{equation}\label{eq1-8}
\begin{cases}
 u_{t}(x,t)-\Delta_{\omega}u(x,t)-f(u(x,t)) & \\
 \geq v_{t}(x,t)-\Delta_{\omega}v(x,t)-f(v(x,t)), & (x,t)\in S\times\left(0,T\right),\\
 u(x,t)\geq v(x,t), & (x,t)\in\partial S\times[0,T),\\
 u\left(x,0\right)\geq v\left(x,0\right), & x\in\overline{S}.
\end{cases}
\end{equation}
If $u\left(x^{*},0\right)>v\left(x^{*},0\right)$ for some $x^{*}\in S$, then $u\left(x,t\right)>v\left(x,t\right)$ for all $(x,t)\in S\times(0,T).$
\end{Thm}

\begin{proof}
First, note that $u\geq v$ on $\overline{S}\times[0,T)$ by above theorem. Let $T'>0$ be arbitrarily given with $T'<T$. Since $f$ is locally Lipschitz continuous on $\mathbb{R}$, there exists $L>0$ such that
\begin{equation}\label{eq1-9}
\left|f\left(a\right)-f\left(b\right)\right|\leq L\left|a-b\right|,\,\,a,b\in[-m,m]
\end{equation}
where $m={\displaystyle \max_{x\in\overline{S}}\max_{0\leq t\leq T'}}\left\{\left|u\left(x,t\right)\right|,\left|v\left(x,t\right)\right|\right\}.$
Let $\tau\,:\,\overline{S}\times\left[0,T'\right]\rightarrow\mathbb{R}$ be the functions defined by
\[
\tau\left(x,t\right) := u\left(x,t\right)-v\left(x,t\right),\,\,\left(x,t\right)\in\overline{S}\times\left[0,T'\right].
\]

Then $\tau\left(x,t\right)\geq0$ for all $\left(x,t\right)\in\overline{S}\times\left[0,T'\right]$. From the inequality \eqref{eq1-8}, we have
\begin{equation}\label{equation0}
\tau_{t}(x^{*},t)-\Delta_{\omega}\tau\left(x^{*},t\right)-\left[f\left(u(x^{*},t)\right)-f\left(v(x^{*},t)\right)\right]\geq0.
\end{equation}
for all $0<t\leq T'$. Using \eqref{eq1-9}, the inequality \eqref{equation0} becomes
\begin{equation*}
\begin{aligned}
  & \tau_{t}\left(x^{*},t\right)\\
  & \geq \sum_{y\in\overline{S}}\left[\tau\left(y,t\right)-\tau\left(x^{*},t\right)\right]\omega\left(x^{*},y\right)+\left[f\left(u\left(x^{*},t\right)\right)-f\left(v\left(x^{*},t\right)\right)\right]\\
  & \geq-d_{\omega}x^{*}\tau\left(x^{*},t\right)-L|\tau\left(x^{*},t\right)|\\
  &=-\left(d_{\omega}x^{*}+L\right)\tau\left(x^{*},t\right).
\end{aligned}
\end{equation*}

This implies
\begin{equation}\label{equation1}
\tau\left(x^{*},t\right)\geq\tau\left(x^{*},0\right)e^{-\left(dx^{*}+L\right)t}>0,\,\,t\in(0,T'],
\end{equation}
since $\tau\left(x^{*},0\right)>0$. Now, suppose there exists $(x_{0},t_{0})\in S\times (0,T']$ such that
\begin{center}
$\tau\left(x_{0},t_{0}\right)=\displaystyle \min_{x\in S,\,0<t\leq T'} \tau\left(x,t\right)=0$.
\end{center}

Then we have
\[
\tau_{t}\left(x_{0},t_{0}\right)\leq0 \hbox{ and } \Delta_{\omega}\tau\left(x_{0},t_{0}\right)\geq0.
\]

Hence, together with inequality \eqref{equation0}, we obtain the following.
\begin{center}
$0\leq\tau_{t}\left(x_{0},t_{0}\right)-\Delta_{\omega}\tau\left(x_{0},t_{0}\right)\leq0$.
\end{center}

Therefore, we have
\[
\Delta_{\omega}\tau\left(x_{0},t_{0}\right)=0
\]
i.e.
$$
\begin{aligned}
&\sum_{y\in\overline{S}}\tau\left(y,t_{0}\right)\omega\left(x_{0},y\right)=0,
\end{aligned}
$$
which implies that $\tau\left(y,t_{0}\right)=0$ for all $y\in \overline{S}$ with $y\sim x_{0}$.
Now, for any $x \in \overline S,$ there exists a path
\begin{displaymath}
x_{0} \sim x_{1} \sim \cdots \sim x_{n-1} \sim x_{n} = x,
\end{displaymath}
since $\overline S$ is connected. By applying the same argument as above inductively we see that $\tau(x, t_{0})=0$ for every $x \in \overline S$. This is a contradiction to \eqref{equation1}. Since $T'<T$ is arbitrarily given, we get $u\left(x,t\right)>v\left(x,t\right)$ for all $\left(x,t\right)\in S\times\left(0,T\right)$.
\end{proof}

We note that by the comparison principle, if $f(0)=0$ then solutions $u$ to the equation \eqref{equation} are positive if initial data $u_{0}$ is nontrivial and nonnegative. On the other hand, it is quite natural that $f$ is positive, when dealing with the blow-up theory. Hence, throughout this paper, we always assume that a function $f$ is locally Lipschitz continuous on $\mathbb{R}$, $f(0)=0$, $f(u)>0$, $u>0$ and the initial data $u_{0}$ is nontrivial and nonnegative.

\section{Blow-Up: Concavity Method}\label{BCM}
In this section, we discuss the blow-up phenomena of the solutions to the equation \eqref{equation} by using concavity method, which is the main part of this paper. This method, introduced by Levine \cite{L}, uses the concavity of an auxiliary function. In fact, the concavity method is an elegant tool for deriving estimates and giving criteria for blow-up.

In order to state and prove our result, we introduce the following new condition: for some $\alpha, \beta, \gamma>0$,
\[
\hbox{(C) $\hspace{1cm} \alpha F\left(u\right) \leq uf(u)+\beta u^{2}+\gamma,\,\,u>0$},
\]
where $0<\beta\leq\frac{\left(\alpha-2\right)\lambda_{0}}{2}$.

\begin{Rmk}

We will discuss the condition (C) in the next section, comparing with the conditions (A) and (B) introduced in the first section, together with the condition $J(0)>0$ for the initial data.
\end{Rmk}

We now state the main theorem of this paper:

\begin{Thm}\label{BlowB}
For the function $f$ with the hypothesis (C), if the initial data $u_{0}$ satisfies
\begin{equation}\label{11}
-\frac{1}{4}\sum_{x,y\in\overline{S}}\left[u_{0}\left(x\right)-u_{0}\left(y\right)\right]^{2}\omega\left(x,y\right)+\sum_{x\in\overline{S}}\left[F(u_{0}\left(x\right))-\gamma\right]>0,
\end{equation}
then the solutions $u$ to the equation \eqref{equation} blow up at finite time $T^{*}$ in the sense of
\[
\lim_{t\rightarrow T^{*}}\int_{0}^{t}\sum_{x\in\overline{S}}u^{2}\left(x,s\right)ds=+\infty,
\]
where $\gamma$ is the constant in the condition (C). \end{Thm}
\begin{proof}
First, we note that $u(x,t)>0$ on $S\times(0,\infty)$, by the strong comparison principle. Now, we define a functional $J$ by
\[
J\left(t\right):=-\frac{1}{4}\sum_{x,y\in\overline{S}}\left[u\left(x,t\right)-u\left(y,t\right)\right]^{2}\omega\left(x,y\right)+\sum_{x\in\overline{S}}\left[F\left(u\left(x,t\right)\right)-\gamma\right],\;\;\;t\geq0.
\]

Then by \eqref{11},
\[
J\left(0\right)=-\frac{1}{4}\sum_{x,y\in\overline{S}}\left[u_{0}\left(x\right)-u_{0}\left(y\right)\right]^{2}\omega\left(x,y\right)+\sum_{x\in\overline{S}}\left[F\left(u_{0}\left(x\right)\right)-\gamma\right]>0.
\]

Multiplying the equation \eqref{equation}  by $u$ and summing up over $\overline{S}$, we obtain from Lemma \ref{Lemma 2.1.} that
\begin{equation}\label{12}
\frac{d}{dt}\sum_{x\in\overline{S}}u^{2}\left(x,t\right)=-\sum_{x,y\in\overline{S}}\left[u\left(x,t\right)-u\left(y,t\right)\right]^{2}\omega\left(x,y\right)+2\sum_{x\in\overline{S}}u\left(x,t\right)f\left(u\left(x,t\right)\right).
\end{equation}

Multiplying the equation \eqref{equation} by $u_{t}$ and summing up over $\overline{S}$, we obtain from Lemma \ref{Lemma 2.1.} that
\[
\sum_{x\in\overline{S}}u_{t}^{2}\left(x,t\right)=-\sum_{x\in\overline{S}}u_{t}\left(x,t\right)\left[-\Delta_{\omega}u\left(x,t\right)\right]+\frac{d}{dt}\left[\sum_{x\in\overline{S}}F\left(u\left(x,t\right)\right)\right]
\]
and
\[
\begin{aligned}
&\sum_{x\in\overline{S}}u_{t}\left(x,t\right)\left(-\Delta_{\omega}u\left(x,t\right)\right)\\
&=\frac{1}{2}\sum_{x,y\in\overline{S}}\left[u\left(y,t\right)-u\left(x,t\right)\right]\left[u_{t}\left(y,t\right)-u_{t}\left(x,t\right)\right]\omega\left(x,y\right)\\
&=\frac{d}{dt}\left[\frac{1}{4}\sum_{x,y\in\overline{S}}\left[u\left(y,t\right)-u\left(x,t\right)\right]^{2}\omega\left(x,y\right)\right].
\end{aligned}
\]

Then it follows that
\begin{equation}\label{13}
\begin{aligned}
\sum_{x\in\overline{S}}u_{t}^{2}\left(x,t\right)=&-\frac{d}{dt}\left[\frac{1}{4}\sum_{x,y\in\overline{S}}\left[u\left(y,t\right)-u\left(x,t\right)\right]^{2}\omega\left(x,y\right)\right]\\
+&\frac{d}{dt}\left[\sum_{x\in\overline{S}}F\left(u\left(x,t\right)\right)\right].
\end{aligned}
\end{equation}

Moreover, it follows from (\ref{13}) that
\[
J'\left(t\right)=\sum_{x\in\overline{S}}u_{t}^{2}\left(x,t\right)
\]
and
\begin{equation}\label{jtt}
J\left(t\right)=\int_{0}^{t}J'\left(s\right)ds+J\left(0\right)=\int_{0}^{t}\sum_{x\in\overline{S}}u_{t}^{2}\left(x,s\right)ds+J\left(0\right).
\end{equation}

Now, we introduce a new function
\[
I\left(t\right)=\int_{0}^{t}\sum_{x\in\overline{S}}u^{2}\left(x,s\right)ds+M,\,t\geq 0,
\]
where $M>0$ is a constant to be determined later. Then it is easy to see that
\begin{equation}\label{15}
\begin{aligned}
I'\left(t\right) & =  \sum_{x\in\overline{S}}u^{2}\left(x,t\right)\\
 & =  \sum_{x\in\overline{S}}\int_{0}^{t}2u\left(x,s\right)u_{t}\left(x,s\right)ds+\sum_{x\in\overline{S}}u_{0}^{2}\left(x\right).
\end{aligned}
\end{equation}

Then we use \eqref{12}, the condition (C), Lemma \ref{Lemma 2.2.}, and \eqref{jtt} in turn to obtain
\begin{equation}\label{14}
\begin{aligned}
\frac{1}{2}I''\left(t\right)&= \frac{1}{2}\frac{d}{dt}\sum_{x\in\overline{S}}u^{2}\left(x,t\right)\\
&=-\frac{1}{2}\sum_{x,y\in\overline{S}}\left[u\left(x,t\right)-u\left(y,t\right)\right]^{2}\omega\left(x,y\right)+\sum_{x\in\overline{S}}u\left(x,t\right)f\left(u\left(x,t\right)\right)\\
&\geq \left[\frac{\alpha-2}{4}-\frac{\alpha}{4}\right]\sum_{x,y\in\overline{S}}\left[u\left(x,t\right)-u\left(y,t\right)\right]^{2}\omega\left(x,y\right) \\
&+\sum_{x\in\overline{S}}\left[\alpha F\left(u\left(x,t\right)\right) - \frac{(\alpha-2)\lambda_{0}}{2}u^{2}\left(x,t\right)-\alpha\gamma\right]\\
&=\alpha\left[-\frac{1}{4}\sum_{x,y\in\overline{S}}\left[u\left(x,t\right)-u\left(y,t\right)\right]^{2}\omega\left(x,y\right)+\sum_{x\in\overline{S}}\left[F\left(u\left(x,t\right)\right)-\gamma\right]\right]\\
&-\frac{(\alpha-2)\lambda_{0}}{2}\sum_{x\in\overline{S}}u^{2}\left(x,t\right)+\frac{\alpha-2}{4}\sum_{x,y\in\overline{S}}\left[u\left(x,t\right)-u\left(y,t\right)\right]^{2}\omega\left(x,y\right)\\
&\geq \alpha\left[J\left(0\right)+\sum_{x\in\overline{S}}\int_{0}^{t}u_{t}^{2}\left(x,s\right)ds\right].
\end{aligned}
\end{equation}

Using the Schwarz inequality, we obtain
\begin{equation}\label{16}
\begin{aligned}
&\left\{ I'\left(t\right)\right\} ^{2}\\
& \leq 4\left(1+\delta\right)\left[\sum_{x\in\overline{S}}\int_{0}^{t}u\left(x,s\right)u_{t}\left(x,s\right)ds\right]^{2}+\left(1+\frac{1}{\delta}\right)\left[\sum_{x\in\overline{S}}u_{0}^{2}\left(x\right)\right]^{2}\\
 & \leq  4\left(1+\delta\right)\left[\sum_{x\in\overline{S}}\left(\int_{0}^{t}u^{2}\left(x,s\right)ds\right)^{\frac{1}{2}}\left(\int_{0}^{t}u_{t}^{2}\left(x,s\right)ds\right)^{\frac{1}{2}}\right]^{2}\\
 &+\left(1+\frac{1}{\delta}\right)\left[\sum_{x\in\overline{S}}u_{0}^{2}\left(x\right)\right]^{2}\\
 & \leq  4\left(1+\delta\right)\left(\sum_{x\in\overline{S}}\int_{0}^{t}u^{2}\left(x,s\right)ds\right)\left(\sum_{x\in\overline{S}}\int_{0}^{t}u_{t}^{2}\left(x,s\right)ds\right)\\
 &+\left(1+\frac{1}{\delta}\right)\left[\sum_{x\in\overline{S}}u_{0}^{2}\left(x\right)\right]^{2},
\end{aligned}
\end{equation}
where $\delta>0$ is arbitrary. Combining the above estimates (\ref{15}), (\ref{14}), and (\ref{16}), we obtain that for $\xi=\delta=\sqrt{\frac{\alpha}{2}}-1>0$,
\begin{equation*}
\begin{aligned}
&I''\left(t\right)I\left(t\right)-\left(1+\xi\right)I'\left(t\right)^{2}\\
 & \geq 2\alpha\left[J\left(0\right)+\sum_{x\in\overline{S}}\int_{0}^{t}u_{t}^{2}\left(x,s\right)ds\right]\left[\sum_{x\in\overline{S}}\int_{0}^{t}u^{2}\left(x,s\right)ds+M\right]\\
 & -  4\left(1+\xi\right)\left(1+\delta\right)\left[\sum_{x\in\overline{S}}\int_{0}^{t}u^{2}\left(x,s\right)ds\right]\left[\sum_{x\in\overline{S}}\int_{0}^{t}u_{t}^{2}\left(x,s\right)ds\right]\\
 & -  \left(1+\xi\right)\left(1+\frac{1}{\delta}\right)\left[\sum_{x\in\overline{S}}u_{0}^{2}\left(x\right)\right]^{2}\\
 & \geq  2\alpha M\cdot J\left(0\right)-\left(1+\xi\right)\left(1+\frac{1}{\delta}\right)\left[\sum_{x\in\overline{S}}u_{0}^{2}\left(x\right)\right]^{2}.
\end{aligned}
\end{equation*}

Since $J\left(0\right)>0$ by assumption, we can choose $M>0$ to be large enough so that
\begin{equation}\label{17}
I''(t)I\left(t\right)-\left(1+\xi\right)I'\left(t\right)^{2}>0.
\end{equation}

This inequality (\ref{17}) implies that for $t\geq0$,
\[
\frac{d}{dt}\left[\frac{I'\left(t\right)}{I^{\xi+1}\left(t\right)}\right]>0  \hbox{  i.e. }I'\left(t\right)\geq\left[\frac{I'\left(0\right)}{I^{\xi+1}\left(0\right)}\right]I^{\xi+1}\left(t\right).
\]

Therefore, it follows that $I\left(t\right)$ cannot remain finite for all $t>0$. In other words, the solutions $u\left(x,t\right)$  blow up in finite time $T^{*}$.
\end{proof}

\begin{Rmk}
The above blow-up time can be estimated roughly. Taking
\[
M:=\frac{\frac{\alpha}{\alpha-2}\left(1+\sqrt{1+\frac{\alpha-2}{2}}\right)\left[{\sum_{x\in\overline{S}}u_{0}^{2}\left(x\right)}\right]^{2}}{2\alpha\left[-\frac{1}{4}{ \sum_{x,y\in\overline{S}}\left[u_{0}\left(x\right)-u_{0}\left(y\right)\right]^{2}\omega\left(x,y\right)+\sum_{x\in\overline{S}}}\left[F\left(u_{0}\left(x\right)\right)-\gamma\right]\right]},
\]
we see that
\[
\begin{cases}
I'\left(t\right)\geq\left[\frac{\sum_{x\in\overline{S}}u_{0}^{2}\left(x\right)}{M^{\xi+1}}\right]I^{\xi+1}\left(t\right), & t>0,\\
I\left(0\right)=M,
\end{cases}
\]
which implies
\[
I\left(t\right)\geq\left[\frac{1}{M^{\xi}}-\frac{\xi\sum_{x\in\overline{S}}u_{0}^{2}\left(x\right)}{M^{\xi+1}}t\right]^{-\frac{1}{\xi}}
\]
where $\xi=\sqrt{\frac{\alpha}{2}}-1>0$. Then the blow-up time $T^{*}$ satisfies
\begin{equation*}\label{22}
0<T^{*}\leq\frac{M}{\xi\sum_{x\in\overline{S}}u_{0}^{2}\left(x\right)}.
\end{equation*}
\end{Rmk}

\section{Discussion on the Condition (C) and $J(0)>0$}\label{(C)}

As seen in the proof of Theorem \ref{BlowB}, the concavity method is a tool for deriving the blow-up solution via the auxiliary function $J(t)$ under the condition (A), (B), or (C), by imposing $J(0)>0$, instead of the large initial data.

In this section, we compare the conditions (A), (B), and (C) each other and discuss the role of $J(0)>0$.

First, let us recall the conditions as follows: for some $\epsilon, \beta, \hbox{ and } \gamma>0$,
\[
\begin{aligned}
&\hbox{(A) $\hspace{1cm} (2+\epsilon) F(u)\leq uf(u)$},\\
&\hbox{(B) $\hspace{1cm} (2+\epsilon) F(u)\leq uf(u) + \gamma$},\\
&\hbox{(C) $\hspace{1cm} (2+\epsilon) F\left(u\right) \leq uf(u)+\beta u^{2}+\gamma$},
\end{aligned}
\]
for every $u>0$, where $0<\beta\leq\frac{\epsilon\lambda_{0}}{2}$ and $F(u):=\int_{0}^{u}f(s)ds$. Here, note that the constants $\epsilon, \beta, \hbox{ and } \gamma>0$ may be different in each case.\\

Then it is easy to see that (A) implies (B) and (B) implies (C), in turn. The difference between (B) and (C) is whether or not they depend on the domain. The condition (B) is independent of the eigenvalue which depends on the domain. However, the condition (C) depends on domain, due to the term $au^{2}$ with $0<a\leq\frac{\lambda_{0}}{2}$. From this point of view, the condition (C) can be understood as a refinement of (B), corresponding to the domain. On the contrary, if a function $f$ satisfies (C) for every domain graph $S$ with boundary, then the eigenvalue $\lambda_{0}$ can be arbitrary small so that the condition (C) get closer to (B) arbitrarily. Besides, as far as the authors know, there has not been any noteworthy condition for the concavity method other than (A) or (B). \\

On the other hand, using the fact that (C) is equivalent to

\begin{equation*}\label{F(u)/u^2+m}
	\frac{d}{du}\left(\frac{F(u)}{u^{2+\epsilon}}-\frac{\gamma}{2+\epsilon}\cdot\frac{1}{u^{2+\epsilon}}-\frac{\beta}{\epsilon}\cdot\frac{1}{u^{\epsilon}}\right)\geq0,\,\,u>0,
	\end{equation*}

we can easily see that for every $u>0$,

\begin{equation}\label{conditionsss}
\begin{aligned}
&\hbox{(A) holds if and only if $F(u)={u^{2+\epsilon}}h_{1}(u)$},\\
&\hbox{(B) holds if and only if $F(u)={u^{2+\epsilon}}h_{2}(u)+b$},\\
&\hbox{(C) holds if and only if $F(u)={u^{2+\epsilon}}h_{3}(u)+au^{2}+b$},
\end{aligned}
\end{equation}
for some constants $\epsilon>0$, $a>0$, and $b>0$ with $0<a\leq\frac{\lambda_{0}}{2}$, where $h_{1}$, $h_{2}$, and $h_{3}$ are nondecreasing function on $(0,+\infty)$. Here also, the constants $\epsilon, a, \hbox{ and } b$ may be different in each case. We note here that the nondecreasing functions $h_{1}$ is nonnegative on $(0,+\infty)$, but  $h_{2}$ and $h_{3}$ may not be nonnegative, in general.


\begin{Lemma}\label{C condition f>=a u^1+e}
	Let $f$ be a function satisfying (C) and $f(u)\geq \lambda u$, $u>0$, where $\lambda>\lambda_{0}$. Then the condition (C) implies that there exists $m>0$ such that $h_{3}(u)>0$ for $u>m$. In this case, we can find $\delta>0$  such that $f(u)\geq \delta u^{1+\epsilon}$, $u\geq m$. Moreover, the conditions (B) and (C) are equivalent.
\end{Lemma}

\begin{proof} First, it follows from \eqref{conditionsss} and the fact $\lambda>\lambda_{0}$ that $F(u)\geq \frac{\lambda}{2}u^{2}\geq \frac{\lambda_{0}}{2}u^{2}$ and so that
\[
u^{2+\epsilon}h_{3}(u)=F(u)-au^{2}-b\geq \frac{\lambda-\lambda_{0}}{2}u^{2}-b,
\]	
which goes to $+\infty$, as $u\rightarrow+\infty$. So, we can find $m>1$ such that $h_{3}(m)>0$, which implies that
\[
F(u)\geq u^{2+\epsilon}h_{3}(u),\,\, u\geq m.
\]
Putting it into the condition (C), we obtain
\[
u^{2+\epsilon}h_{3}(m)\leq uf(u) + \beta u^{2} + \gamma
\]
or
\[
u^{1+\epsilon}h_{3}(m)\leq f(u)  + \beta u+\frac{\gamma}{u}\leq (1 + \frac{\epsilon}{2})f(u) + \gamma,\,\,u\geq m>1,
\]
which gives
\[
f(u)\geq \delta u^{1+\epsilon},\,\,u\geq m>1
\]
for some $\delta>0$ and another constant $m$. \\
	
	Now, assume that the condition (C) is true. Since $0<\beta\leq\frac{\epsilon\lambda_{0}}{2}$ and $f(u)\geq \lambda u>\lambda_{0}u$, $u>0$, it follows from (C) that
	\[
	\epsilon_{1} F\left(u\right) + \left(2+\epsilon_{2}\right) F\left(u\right) \leq uf(u)+\frac{\epsilon\lambda_{0}}{2} u^{2}+\gamma,
	\]
	where $\epsilon_{1}=\frac{\epsilon\lambda_{0}}{\lambda}>0$ and $\epsilon_{2}=\epsilon-\epsilon_{1}>0$. This implies that for every $u>0$,
	\[
	\begin{aligned}
	uf(u)+\gamma &\geq \left(2+\epsilon_{2}\right) F\left(u\right) + \epsilon_{1} \int_{0}^{u}\left[f\left(s\right)-\lambda s\right]ds\\
	&\geq \left(2+\epsilon_{2}\right) F\left(u\right),
	\end{aligned}
	\]
	which gives (B).

\end{proof}

\begin{Rmk}
It is well known that if $\int_{m}^{+\infty}\frac{ds}{f(s)}=+\infty$ for some $m>0$, the solutions to equation \eqref{equation} is global. On the contrary, it has not been clear yet whether or not the condition $\int_{m}^{+\infty}\frac{ds}{f(s)}<+\infty$ guarantees the blow-up solution. Instead, when $f(u)\geq \delta u^{1+\epsilon}$, $u\geq m$ for some $\epsilon>0$ and $m>0$, the solutions to the equation \eqref{equation} blow up in a finite time, only if the initial data $u_{0}$ is sufficiently large i.e. $\max_{x\in S} u_{0}(x)\geq \max\left[\left(\frac{\omega_{0}}{\delta}\right)^{1/\epsilon}, m\right]$, where $\omega_{0}=\max_{x\in S}d_{\omega}x$ (for more details, see \cite{CC}).\\
\end{Rmk}

In general, only the condition (C) may not guarantee the blow-up solutions for any initial data $u_{0}$. In fact, we can easily see that a linear function $f(u)=au$ satisfies (C) if and only if $a\leq \lambda_{0}$. However, for any function $u_{0}$,
		\[
		\begin{aligned}
		J\left(0\right) & =  -\frac{1}{4}\sum_{x,y\in\overline{S}}\left[u_{0}\left(x\right)-u_{0}\left(y\right)\right]^{2}\omega\left(x,y\right)+\sum_{x\in\overline{S}}\left[\frac{a}{2}u_{0}^{2}\left(x\right)-\gamma\right]\\
		& \leq  \frac{1}{2}\left[a-\lambda_{0}\right]\sum_{x\in\overline{S}}u_{0}^{2}\left(x\right)-\gamma|\overline{S}|\,<\,0,
		\end{aligned}
		\]
		which means that there is no initial data $u_{0}$ satisfying $J(0)>0$, when $f(u)=au$, $a\leq\lambda_{0}$. Of course, it is well known that the solutions to \eqref{equation} is global, in this case.

So, from now on, we are going to discuss when we can find initial data $u_{0}$ satisfies $J(0)>0$.
\begin{Lemma}
	Let $f$ satisfy the condition (C). If there exists $v_{0}>0$ such that $F(v_{0})>\omega_{0}v_{0}^{2}+\gamma_{1}$, where $\gamma_{1}\geq\frac{\gamma|\overline{S}|}{|S|}$, then there exists the initial data $u_{0}$ such that $J(0)>0$. Here, $\omega_{0}=\max_{x\in S}d_{\omega}x$.
\end{Lemma}

\begin{proof}
	Since $F$ is continuous on $[0,+\infty)$, there exist $a,b>0$ with $0<a<b$ such that $F(v)>\omega_{0}v^{2}+\gamma_{1}$, $v\in(a,b)$. Then for every function $u_{0}(x)$ satisfying
	\[
	\begin{cases}
		u_{0}\left(x\right)\in (a,b), & x\in S,\\
		u_{0}|_{\partial S}=0.
	\end{cases}
	\]
	It follows that
	\[
	\begin{aligned}
	J\left(0\right) & =  -\frac{1}{4}\sum_{x,y\in\overline{S}}\left[u_{0}\left(x\right)-u_{0}\left(y\right)\right]^{2}\omega\left(x,y\right)+\sum_{x\in\overline{S}}\left[F(u_{0}\left(x\right))-\gamma\right]\\
	& \geq  -\frac{1}{2}\sum_{x,y\in\overline{S}}\left[u_{0}^{2}\left(x\right)+u_{0}^{2}\left(y\right)\right]\omega\left(x,y\right)+\sum_{x\in\overline{S}}\left[F(u_{0}\left(x\right))-\gamma\right]\\
	& = -\sum_{x\in\overline{S}}u_{0}^{2}\left(x\right)d_{\omega}x+\sum_{x\in\overline{S}}[F(u_{0}\left(x\right))-\gamma]\\
	& \geq \sum_{x\in\overline{S}}\left[F(u_{0}\left(x\right))-\omega_{0}u_{0}^{2}\left(x\right)-\gamma\right]\\
	& = \sum_{x\in\overline{S}}\left[F(u_{0}\left(x\right))-\omega_{0}u_{0}^{2}\left(x\right)\right]-\gamma|\overline{S}|\\
	&>\gamma_{1}|S|-\gamma|\overline{S}|\geq0
	\end{aligned}
	\]
	Therefore, $J(0)>0$.
\end{proof}
\begin{Cor}
	\begin{enumerate}[(i)]
		\item 	If there exists $(a,b)$ such that $F(v)>\omega_{0}v^{2}+\gamma_{1}$, $\gamma_{1}\geq \frac{\gamma|\overline{S}|}{|S|}$ for every $v\in(a,b)$, then for every $u_{0}$ with
		\[
		\begin{cases}
		u_{0}\left(x\right)\in (a,b), & x\in S,\\
		u_{0}|_{\partial S}=0,
		\end{cases}
		\]
		we see that $J(0)>0$.
		\item If $F(v)>\omega_{0}v^{2}+\gamma_{1}$, $\gamma_{1}\geq\frac{\gamma|\overline{S}|}{|S|}$ for every $v\in(0,+\infty)$, then the solutions blow up for every initial data $u_{0}>0$.
	\end{enumerate}
\end{Cor}


\section*{Conflict of Interests}
\noindent The authors declare that there is no conflict of interests regarding the publication of this paper.

\section*{Acknowledgments}
\noindent This work was supported by Basic Science Research Program through the National Research Foundation of Korea(NRF) funded by the Ministry of Education (NRF-2015R1D1A1A01059561). In addition, the authors would like to express thanks to anonymous reviewers for their excellent suggestions.

\end{document}